\documentclass[11pt]{amsart}
\usepackage{amssymb,amsthm,amsmath,amstext}
\usepackage{mathrsfs}  
\usepackage{bm}        
\usepackage{mathtools} 
\usepackage{color}
\usepackage[bbgreekl]{mathbbol}
\usepackage{multirow}
\usepackage{enumitem}
\usepackage{cite}

\usepackage[left=3cm,right=3cm]{geometry}

\usepackage{graphicx}

\usepackage[all]{xy}
\usepackage{tikz}
\usepackage{tikz-cd}

\theoremstyle{plain}
\newtheorem{theorem}{Theorem}[section]

\newtheorem{proposition}[theorem]{Proposition}
\newtheorem{lemma}[theorem]{Lemma}
\newtheorem{corollary}[theorem]{Corollary}

\newtheorem{question}[theorem]{Question}

\numberwithin{equation}{section}

\theoremstyle{definition}
\newtheorem{definition}[theorem]{Definition}

\newtheorem{example}[theorem]{Example}
\newtheorem{conjecture}[theorem]{Conjecture}

\theoremstyle{remark}
\newtheorem{remark}[theorem]{Remark}

\newcommand{\C}{\mathbb C}
\renewcommand{\P}{\mathbb P}

\DeclareMathOperator{\SL}{SL}
\DeclareMathOperator{\tr}{tr}

\newcommand{\id}{\mathrm{id}}
\newcommand{\diag}{\mathrm{diag}}

\newcommand{\GL}{\mathrm{GL}}

\newcommand{\quadand}{\quad \text{ and } \quad}

\usepackage[backref=page]{hyperref}

\newcommand{\bP}{\mathbb{P}}

\newcommand{\bQ}{\mathbb{Q}}
\newcommand{\bZ}{\mathbb{Z}}

\newcommand{\bC}{\mathbb{C}}

\newcommand{\calA}{\mathcal{A}}
\newcommand{\calC}{\mathcal{C}}

\newcommand{\calM}{\mathcal{M}}

\newcommand{\PGL}{\mathrm{PGL}}
\newcommand{\Gr}{\mathrm{Gr}}

\newcommand{\git}{/\kern-0.2em/}

\newcommand{\tH}{\widetilde{\mathrm{H}}}

\newcommand{\lisa}[1]{{\color{purple} \sf $\clubsuit\clubsuit\clubsuit$ Lisa: [#1]}}

\begin{document}

\title[Fourier-Mukai partners]{Fourier-Mukai partners of non-syzygetic cubic fourfolds and Gale duality}

\author[B\"ohning]{Christian B\"ohning}\thanks{For the purpose of open access, the authors have applied a Creative Commons Attribution (CC BY) licence to any Author Accepted Manuscript version arising from this submission.}
\address{Christian B\"ohning, Mathematics Institute, University of Warwick\\
Coventry CV4 7AL, England}
\email{C.Boehning@warwick.ac.uk}

\author[von Bothmer]{Hans-Christian Graf von Bothmer}
\address{Hans-Christian Graf von Bothmer, Fachbereich Mathematik der Universit\"at Hamburg\\
Bundesstra\ss e 55\\
20146 Hamburg, Germany}
\email{hans.christian.v.bothmer@uni-hamburg.de}

\author[Marquand]{Lisa Marquand}
\address{Lisa Marquand, Rutgers University\\
Hill Center, Busch Campus\\
110 Frelinghuysen Road\\
Piscataway\\
NJ 8019, USA}
\email{lisa.marquand@rutgers.edu}

\date{\today}


\begin{abstract}
We study so-called non-syzygetic cubic fourfolds, i.e., smooth cubic fourfolds containing two cubic surface scrolls in distinct hyperplanes with intersection number between the two scrolls equal to $1$. We prove that a very general non-syzygetic cubic fourfold has precisely one nontrivial Fourier-Mukai partner; further, we prove that this partner is also non-syzygetic. We characterise non-syzygetic cubic fourfolds algebraically as those having a special type of equation that is almost linear determinantal, and show that the equation of the Fourier-Mukai partner can be obtained by applying Gale duality. We establish that Gale dual cubics are birational, Fourier-Mukai partners and have birational Fano varieties of lines under suitable genericity assumptions, recovering a result of Brooke-Frei-Marquand. We show that the birationality of the Fano varieties of lines continues to hold in the context of equivariant birational geometry, but birationality of the cubics may not. We exhibit examples of Gale dual cubics with faithful actions of the alternating group on four letters that could provide counterexamples to equivariant versions of a conjecture by Brooke-Frei-Marquand predicting birationality of the cubics if the Fano varieties of lines are birational, and also possibly a related conjecture by Huybrechts predicting birationality of Fourier-Mukai partners. 
\end{abstract}

\maketitle 

\section{Introduction}\label{sIntroduction}
One of the well-known problems shaping the field in higher-dimensional birational geometry is to exhibit accessible algebraic, categorical or geometric structures that effectively capture information about the birational type of smooth cubic fourfolds. In \cite{kuzcubic}, Kuznetsov proposed extracting birational information about a smooth cubic fourfold $X$ from a full triangulated subcategory $\calA_X$ of the derived category of coherent sheaves, often called the Kuznetsov component of $X$. Building on this, the following conjecture was proposed in \cite[Chapter 7, Conjecture 3.21]{Huybrechtscubicsbook} and \cite[Conjecture 2.5]{Huybrechts:2025aa}. We let $X, X'$ be smooth cubic fourfolds over $\bC$.

\begin{conjecture}[Huybrechts' conjecture]  \label{conj: Huy}
    Let $\mathcal{A}_X$ and $\mathcal{A}_{X'}$ be the Kuznetsov components in the standard semiorthogonal decompositions in $\mathrm{D}^b (X) $ and $\mathrm{D}^b (X')$.
 If $X$ and $X'$ are \emph{Fourier-Mukai partners}, i.e. $\mathcal{A}_X\simeq \mathcal{A}_{X'}$, then $X$ and $X'$ are birational. 
\end{conjecture}
There are few known examples of pairs of cubic fourfolds $X,X'$ that satisfy Conjecture \ref{conj: Huy}. Some evidence was provided by Fan and Lai \cite{FanLaiCremona}, who considered the case of a very general cubic $X$ containing a Veronese surface. They constructed a birational map from $X$ to its unique nontrivial Fourier Mukai partner $X'$ via a Veronese duality construction. Many known derived equivalences for $\mathcal{A}_X$ follow this trend; the constructions rely on  linear algebra and duality. Examples include the equivalence between $\calA_X$ for a Pfaffian cubic fourfolds and the derived category of its associated $K3$ surface, (see \cite{BD85}, \cite{KuznetsovHPDGrassmannians}), or the equivalence of Kuznetsov components between a cubic fourfold containing a cubic scroll and a Gushel-Mukai containing a certain type of plane \cite{KuzPerry}.

More recently, in \cite{Brooke:2024aa} (see also \cite{BFMQ24}) the authors study cubic fourfolds containing two cubic surface scrolls $T_1$ and $T_2$ in distinct hyperplanes with $[T_1].[T_2]=1$. Such a cubic is said to be \textit{non-syzygetic}. 
They prove that if $X$ is very general with this property, then there exists another non-syzygetic cubic fourfold $X'$ such that the pair $X$ and $X'$ are not isomorphic, but satisfy Conjecture \ref{conj: Huy}. 

The main goal of this work is to make this Fourier-Mukai partner construction completely explicit using techniques from multilinear and commutative algebra. More precisely, given an equation defining one of the cubic hypersurfaces, one can directly write down the equation of its Fourier-Mukai partner. This was entirely beyond the reach of the transcendental techniques (relying in particular on Torelli) employed in \cite{Brooke:2024aa}. We also establish an interesting connection between this duality construction and a duality from convex geometry and combinatorics, namely Gale duality. The explicit description of the Fourier-Mukai partners has already found an application in \cite{BGvBM25} where it is indispensable to enumerate the $1120$ nontrivial Fourier-Mukai partners of  certain cubic fourfolds with involutions. 

Our main result is as follows:

\begin{theorem}
    Let $X$ be any smooth non-syzygetic cubic fourfold. Then $X$ has an equation of the form
    $$\det(M)+L_1L_2L_3=0,$$ where $M$ is a $3\times 3$ matrix of linear forms, and $L_1,L_2, L_3$ are linear forms with at least two linearly independent.
    Moreover, consider the linear map $$\C^{12} \xrightarrow{(M,L_1,L_2,L_3)} \C^6,$$ with kernel $$\C^6\xrightarrow{(M',L_1', L_2', L_3')^t} \bC^{12}.$$  Let $X'$ be the cubic fourfold defined by $$\det(M')-L'_1L'_2L'_3=0.$$ Then:
    \begin{enumerate}
        \item For general choices, $X'$ is birational (but not isomorphic) to $X$,
        \item $X'$ is a nontrivial Fourier Mukai partner of $X$.
        \item For very general $X$, we have that $X'$ is the unique nontrivial Fourier Mukai partner of $X$.
    \end{enumerate}
\end{theorem}
We call $X'$ the Gale dual of $X$. 
As one corollary, we recover \cite[Thm. 3.1]{Brooke:2024aa} in a much more explicit and precise manner.

The strategy of proof is as follows: first we show that a pair of Gale dual cubics $X, X'$ as above plus a choice of linear form $L_i$ determines a linear algebra datum that we call a $\rho$-Lagrangian data set, see Definition \ref{dRhoLagrangian}. 
Conversely, every such $\rho$-Lagrangian data set determines a pair of Gale dual cubic fourfolds. 
In Theorem \ref{tConnectionToGMs} we establish that such data sets parametrise EPW-sextics that contain two disjoint rational surfaces isomorphic to $\bP^2$.
An EPW-sextic parametrises Gushel-Mukai fourfolds that are determined by the same data sets. Choosing a point in one of the rational surfaces corresponds to a Gushel Mukai fourfold containing a $\rho$-plane (these are images of planes of the form $\Gr(2,V_3)$), as in \cite{KuzPerry}. 
Choosing such Gushel-Mukai fourfolds and projecting from the corresponding $\rho$-planes recovers the Gale dual cubics $X$ and $X'$, see Theorem \ref{tConnectionToGMs}, b). 
This establishes that $X$ and $X'$ are Fourier-Mukai partners and birational using results from \cite{KuzPerry}, \cite{KuzPerryCatCones}, \cite{DK1}.
For a very general non-syzygetic cubic fourfold, we also compute the number of Fourier-Mukai partners to conclude the uniqueness statement.

In \cite{Brooke:2024aa}, the authors observe that for a non-syzygetic cubic fourfold and its Fourier-Mukai partner, the Fano varieties of lines $F(X)$ and $F(X')$ are birational. This leads to the following conjecture:

\begin{conjecture}\cite[Conjecture 1.2]{Brooke:2024aa} \label{conj: BFM}
If $X$ and $X'$ have birationally equivalent Fano varieties of lines $F(X)$ and $F(X')$, then $X$ and $X'$ are birationally equivalent. 
\end{conjecture}

For Gale cubics $X$ and $X'$, we also show that $F(X)$ and $F(X')$ are birational using an explicit construction (see Corollary \ref{cBirFanos}), which  only depends on the given $\rho$-Lagrangian data sets, as do the Gale dual cubics.

One benefit of a linear algebraic construction is that one can easily study an equivariant version - we can endow the Gale dual cubics $X$ and $X'$ with faithful actions of a finite group $G$ and ask if our results hold $G$-equivariantly. 
Indeed, the search for candidate counterexamples to the equivariant versions of the above conjectures served as our starting point for this work. 
For Gale dual cubics, the birationality between $F(X)$ and $F(X')$ continues to hold equivariantly. 
However, the birationality of $X$ and $X'$ depends on a choice of Gushel-Mukai fourfolds which may not exist $G$-equivariantly. We exploit this in Section \ref{sGActions} where we present candidate counter-examples to the equivariant versions of the conjectures above. 
The locus of nonsyzygetic cubic fourfolds contains many examples of cubic fourfolds with automorphisms (see \cite[Sect. 4]{BGM25} for examples).

Here is a more detailed road map of the paper. 
In Section \ref{sNonsyz} we show that non-syzygetic cubic fourfolds can be characterised algebraically as those cubic fourfolds that admit a special kind of equation as above, cf. Proposition  \ref{pNonsyzygetic}. 

In Section \ref{sCounting} we prove that a very general non-syzygetic cubic fourfold $X$ has precisely one nontrivial Fourier-Mukai partner $X'$. The argument is an adaptation of the method used by Fan and Lai in \cite{FanLaiCremona} and reduces the question to a certain counting problem for overlattices. 

In Section \ref{sLagrangian} we define the \emph{Gale dual cubic} $X'$ of a non-syzygetic cubic $X$ more precisely. We establish the connection between Gale dual cubics and certain Lagrangian construction data that we call $\rho$-Lagrangians. This will be used later to establish birational links of Gale dual cubics to Gushel-Mukai fourfolds containing $\rho$-planes \cite[\S 5]{KuzPerry}, in particular proving that $X$ and $X'$ are FM partners and birational. 

In Section \ref{sRepTheory} we recast the previous construction in terms of representation and invariant theory. 

In Section \ref{sFanoLines} we show that the Fano varieties of lines of Gale dual cubics $X$ and $X'$ are birational under certain genericity assumptions. The point here is that our construction is completely explicit and also carries over to the equivariant setup. 

In Section \ref{sGM} we show that $X$ and $X'$ are birational to Gushel-Mukai fourfolds that are period partners. In fact, $X$ and $X'$ can be obtained from Gushel-Mukai fourfolds $Z$ and $Z'$ containing $\rho$-planes by projecting from these $\rho$-planes. The choice of $Z$ and $Z'$ is not unique here but there is a $\bP^2$ worth of choices for each of them, corresponding to points in two disjoint planes contained in an EPW sextic. An important point to note here is that if $X$ and $X'$ are constructed from $G$-Lagrangian construction data, in general no such $Z$ and $Z'$ need to exist as $G$-varieties, and the birational maps between $X$ and $X'$ constructed by \cite{KuzPerry} and \cite{DK1} are not $G$-birational. 

We exploit this in Section \ref{sGActions} where we construct smooth $A_4$-Gale dual cubics $X$ and $X'$ that have $A_4$-birational Fano varieties of lines, but where existing constructions in the literature did not allow us to establish any $A_4$-birational maps between $X$ and $X'$. In fact, we believe these could be candidate counterexamples to the equivariant version of the BFM conjecture above, and also possibly the equivariant version of Huybrechts' conjecture. Moreover, in Proposition \ref{prop:FanoLinesBirational} we show in general that, if $X$ and $X'$ are endowed with $G$-actions in the natural way by starting from $G$-Lagrangian construction data, then their Fano varieties of lines are $G$-birational.

\medskip

\textbf{Acknowledgments.} We would like to thank Yuri Tschinkel and Zhijia Zhang for their interest in and comments on this work. We would like to thank Corey Brooke for his comments on an earlier version. The first and last named authors would like to thank both the Simons Foundation and New York University for excellent working conditions during the completion of this article. Finally, we would like to thank the referee for their very useful comments and suggestions.

For the purpose of open access, the first author has applied a Creative Commons Attribution (CC-BY) licence to any Author Accepted Manuscript version arising from this submission.

\section{Non-syzygetic cubic fourfolds}\label{sNonsyz}

Non-syzygetic cubic fourfolds can be introduced in different ways. We refer to \cite[Section 2, in particular Def. 2.9]{BFMQ24} for more on this. Here we adopt the following definition. 

\begin{definition}\label{dNonsyz}
A smooth cubic fourfold $X$ is called non-syzygetic if it contains cubic surface scrolls $T_1, T_2$ contained in distinct hyperplanes such that $T_1\cdot T_2=1$.
\end{definition}

We denote by $\calC_{nonsyz}$ the locus of non-syzygetic cubic fourfolds inside the moduli space of cubic fourfolds. This is codimension one inside the Hassett divisor $\calC_{12}.$

For us, the following algebraic characterisation of this class of cubics is more important in the sequel. 

\begin{proposition}\label{pNonsyzygetic}
Every non-syzygetic cubic fourfold $X$ is given by an equation of the form 
\begin{gather}\label{fNonsyzEquation}
 \det (M) + L_1L_2L_3 =0
\end{gather}
where $M$ is a $3\times 3$ matrix of linear forms and $L_1, L_2, L_3$ are some further linear forms.

Conversely, if a cubic fourfold defined by an equation of the preceding form is smooth and at least two of  $L_1, L_2, L_3$ are linearly independent, then $X$ is non-syzygetic. 
\end{proposition}

\begin{proof}
If $X$ is non-syzygetic, it contains cubic scrolls $T_1, T_2$ in distinct hyperplanes $L_1=0$, $L_2=0$ such that $T_1\cdot T_2=1$. The scroll $T_i$ is defined by the maximal minors of a $2\times 3$ matrix of linear forms $N_i$ and the hyperplane $L_i$. The equation $F$ of $X$ is in the ideal of $T_i$, hence one can find a $1\times 3$ matrix of linear forms $N_i'$ and a quadric $Q_i$ such that 
\[
F = \det M_i + Q_i L_i
\]
where $M_i$ is the $3\times 3$ matrix 
\[
M_i =\begin{pmatrix} N_i'\\ N_i \end{pmatrix}.
\]
The two by two minors of every two linearly independent generalised rows of $M_i$, together with $L_i$, define a cubic scroll on $X$, hence one gets a $\mathbb{P}^2$ worth of such scrolls. 

\medskip

Now we restrict $F$ to $L_1=L_2=0$ and we want to show that after invertible row and column operations 
\[
(\ast) \quad M_1\mid_{\{L_1=L_2=0\}} = M_2\mid_{\{L_1=L_2=0\}}. 
\]
The cubic scrolls $T_1, T_2$ restrict to twisted cubic curves $C_1, C_2$ on the cubic surface $S$ defined by $F$ restricted to $L_1=L_2=0$, and $C_1\cdot C_2=1$ by assumption. The matrices $M_i\mid_{\{L_1=L_2=0\}}$ give two nets of such curves on $S$, and conversely these nets define the matrices up to row and column operations. Thus we need to show that these nets are the same. Without loss of generality, we can assume that $C_1$ and $C_2$ are general in their nets. It is known that one can realise $S$ as a blowup of $\mathbb{P}^2$ in six points such that $C_1$ is the image of a line in $\mathbb{P}^2$ not passing through any of the blowup points. The image $C_2'$ of $C_2$ in $\mathbb{P}^2$ intersects the line in one point, hence it is also a line. Therefore $C_1$ and $C_2$ generate the same net. 

\medskip

After invertible row and column operations on $M_1$ we can thus assume $(\ast)$. Then there exists a matrix $M$ such that $M\mid_{\{L_i=0\}}=M_i$, $i=1, 2$. We can choose an equation $F$ for $X$ such that 
\[
F\mid_{\{L_1=L_2=0\}} = \det M\mid_{\{L_1=L_2=0\}} .
\]
 Then it holds that 
 \[
 F\mid_{\{L_i =0\}}= \det M\mid_{\{ L_i=0\}}.
 \]
 A priori this only holds up to a factor, but because of the preceding equation we have equality. Hence $F- \det (M)$ vanishes on $L_1=0$ and $L_2=0$, thus there is a third linear form such that
 \[
F = \det (M) + L_1L_2L_3 .
\]
 \end{proof}

\section{Counting Fourier-Mukai partners of non-syzygetic cubic fourfolds}\label{sCounting}

Following the general method used in \cite{FanLaiCremona} adapted to our situation, we will prove
\begin{theorem}\label{tFMPartners}
    Let $X\in \calC_{nonsyz}$ be a very general cubic fourfold. Then there exists a unique non-trivial Fourier-Mukai partner $X'$ of $X$. 
\end{theorem}

\begin{remark}
    After the writing of this article, we developed a general procedure to count the number of Fourier-Mukai partners for a cubic fourfold with a given primitive algebraic lattice and transcendental Hodge structure, see \cite{BGvBM25}. One can apply the procedure there to achieve a slightly more streamlined proof of Theorem \ref{tFMPartners}. We choose to keep the original count here as it inspired our subsequent work, and it is helpful to see the algorithm in a small rank example.
\end{remark}

We denote by $\mathrm{FM}(X)$ the set of isomorphism classes of Fourier-Mukai partners of $X$.  For an even lattice $L$, we denote by $D(L):=L^\ast/L$ the discriminant group of $L$, with associated quadratic form $q_L:D(L)\rightarrow \bQ/2\bZ.$ We write $\tH_X:= \tH(\calA_X,\bZ)$ for the Addington-Thomas Hodge structure of $X$.

\medskip
The primitive cohomology $H^4(X,\bZ)_{prim}$ of any cubic fourfold $X$ is isomorphic as an abstract integral lattice to $L:=U^2\oplus E_8^2\oplus A_2$, where $U$ is the hyperbolic lattice, and $E_8, A_2$ are the standard positive definite ADE lattices.
For a very general non-syzygetic cubic $X$, the algebraic lattice $A(X):=H^4(X,\bZ)\cap H^{2,2}(X)$ is spanned by the classes $\eta_X:=[h^2]$, where $[h]$ is the class of a hyperplane section, and the classes $[T_1], [T_2]$ of the cubic surface scrolls from Definition \ref{dNonsyz}, using \cite[Thm. 1.1]{laz09}. It is easy to work out the intersection matrix with respect to this basis:
\[\left(\begin{matrix}
    3 & 3& 3\\
    3 & 7 & 1\\
    3 & 1 & 7
\end{matrix}\right),\] with determinant $36$. The primitive cohomology is then spanned by $\eta_X-T_1, \eta_X-T_2$. Hence we see that $A(X)_{prim}\cong A_2(2)$. It follows that the transcendental cohomology $T(X):=A(X)_{prim}^\perp\subset H^4(X,\bZ)_{prim}$ is an even lattice of signature $(18,2)$ and discriminant group $D(T(X))=(\bZ/2\bZ)^2\times (\bZ/3\bZ)^2.$

\begin{remark}\label{remark: q for T}
Note that the quadratic form $q_{T(X)}:D(T(X))\rightarrow \bQ/2\bZ$ satisfies $$q_{T(X)}=-q_{A_2(2)}\oplus -b_{\langle \eta\rangle}.$$ In other words, $q_{T,3}\cong \langle \frac{2}{3}\rangle^{\oplus 2}$. Indeed,  we see $A(X)\cong \langle \eta\rangle \oplus A_2(2)$ and $H^4(X,\bZ)$ is a unimodular overlattice of $A(X)\oplus T(X)$. We have that $D(\eta)=\bZ/3\bZ$ with generator $\frac{\eta}{3}.$ It follows that $q(\alpha)=-b(\frac{\eta}{3},\frac{\eta}{3})=-\frac{1}{3}$, where $\alpha$ generates the copy of $\bZ/3\bZ$ orthogonal to the image of $D(A_2(2))$ inside of $D(T(X)).$ 
\end{remark}
\medskip

First, note the following:

\begin{lemma}
    Let $X\in \calC_{nonsyz}$ be very general. Let $Y\in \mathrm{FM}(X)$. Then $Y\in \calC_{nonsyz}.$
\end{lemma}
\begin{proof}
    Any Fourier-Mukai transform beween $X$ and $Y$ induces an isomorphism $\tH_X\simeq \tH_Y$ between Addington-Thomas Hodge structures \cite[Proposition 3.4]{Huy17}, in turn giving a Hodge isometry between the transcendental lattices $T:=T(X)$ and $T(Y)$.  It follows that $A(Y)_{prim}$ is the orthogonal complement of $T$ under some primitive embedding $T\hookrightarrow L:=U^\oplus E_8^2\oplus A_2.$

    We compute the orthogonal complements of all primitive embeddings $\iota:T\hookrightarrow L$. Let $K:=\iota(T)^\perp$. One can easily compute this using OSCAR \cite{OSCAR}, \cite{OSCAR-book}, but we include the proof for completeness.
    
    By Nikulin's classification of primitive embeddings \cite[Proposition 1.15.1]{nikulin}, $K$ is determined by the glueing group: subgroups $H_T\subset D(T)$, $H_L\subset D(L)$ such that $H_T, H_L$ are isometric. Since $D(L)=\bZ/3\bZ$, we have $H_L=0$ or $H_L=D(L).$

    If $H_L=0,$ then the 3-primary part of $D(K)$ is isomorphic to $(\bZ/3\bZ)^3$: this is impossible since $K$ has rank 2.

    Thus $H_L=D(L),$ and $H_T=\bZ/3\bZ$. Since it is isometric to $D(L),$ $H_T$ must be one of the two coordinate copies of $\bZ/3\bZ$, where $q(\alpha)= \frac{2}{3}$ for a generator $\alpha.$ It follows that $|D(K)|=12,$ and $q_K\cong q_{A_2(2)}$, hence $K$ lies in the genus of $A_2(2).$ Since $A_2(2)$ is unique in its genus using \cite[Lemma A.7]{Mar23} it follows that $K\cong A_2(2)$ for any primitive embedding $\iota$, and hence $A(Y)_{\mathrm{prim}} \simeq A_2(2)$, too. A very general such cubic is non-syzygetic by Lemma \ref{lScrollRep}. 
\end{proof}
\begin{lemma}\label{lScrollRep}
    Let $X$ be a cubic fourfold with $A(X)_{\mathrm{prim}}=A_2(2).$ Then $X\in \calC_{nonsyz}.$
\end{lemma}
\begin{proof}
    Let $\{\alpha_1, \alpha_2\}$ be a basis of $A(X)_{\mathrm{prim}}=A_2(2),$ and let $\tau_i=\eta_X+\alpha_i\in A(X).$ Then $\tau_i^2=7, \eta_X\cdot \tau_i=3$ and $\tau_1\cdot \tau_2=1.$ By applying the argument of \cite[Lemma 2.11]{BFMQ24}, the classes $\tau_1, \tau_2$ are represented by cubic scrolls. They form a non-syzygetic pair by above.
\end{proof}

\begin{remark}\label{remark: symm}
    Set $\alpha_3=-\alpha_1-\alpha_2$. Then $\tau_3:=\eta_X+\alpha_i$ is also represented by a cubic scroll. Indeed, for a non-syzygetic cubic fourfold with Equation \ref{fNonsyzEquation}, a representative is contained in the hyperplane section  $\{L_3=0\}\cap X$. There is a natural $S_3$ action on the lattice $A(X)_{prim}$, induced from permuting the classes $\tau_1, \tau_2, \tau_3$. The central involution exchanges the triples $\alpha_1,\alpha_2, \alpha_3$ with their negatives. Thus $O(A(X)_{prim})=S_3\times \{\pm id_S\}.$
\end{remark}

\lisa{I have simplified the following by removing some of the non necessary notations. I will also add a remark referencing the new counting method.}
Fix a very general $X\in \calC_{nonsyz}$ and let 
$$T:=T(X), \text{ and } S:= A(X)_{prim}=A_2(2).$$

We define $\calM_{S,T}$ to be the collection of even overlattices $L\supset S\oplus T$ that satisfy:
\begin{enumerate}
    \item $S\oplus T\subset L\subset S^*\oplus T^*;$
    \item $S,T$ both saturated in $L$;
    \item $L$ has discriminant 3.
\end{enumerate}
Then $L\cong H^4(X,\bZ)_{prim}$ as abstract lattices, but the Hodge structures for each $L\in \calM_{S,T}$ will be different. 

Recall that for a very general $X\in \calC_{nonsyz}$, the Hodge structure $T(X),$ and thus $T$, only admits $\pm \id_T$ as Hodge isometries (see for example \cite[Section 14]{BFvGS}). On the other hand, the isometry group of $S$ is $$O(S)\cong S_3\times \bZ/2\bZ,$$ where $S_3\cong W(A_2)$ permutes the roots of $S\cong A_2(2)$ and $\bZ/2\bZ$ acts as $\pm \id_{S},$ as in Remark \ref{remark: symm}

Assume $Y\in \mathrm{FM}(X)$. Then there exists a Hodge isometry $\psi: T\rightarrow T(Y)= : T_Y$ (there are thus two such choices).
Since $Y\in \calC_{nonsyz}$, there exists an isometry $$\phi:S\rightarrow S_Y:=A(Y)_{prim},$$ with $\phi\in O(S).$
These induce an isometry on the dual spaces:
$$\phi^*\oplus \psi^*: S_Y^*\oplus T_Y^*\rightarrow S^*\oplus T^*.$$

Note that $L_Y:=H^4(Y,\bZ)_{prim}$ is an element of the set $\calM_{S_Y, T_Y}$. We define
$$L_{Y,\phi, \psi}:=(\phi^*\oplus \psi^*)(L_Y).$$
We see that $L_{Y,\phi,\psi}\in \calM_{S,T}$. 

Let 
$$\widetilde{\mathrm{FM}(X)}:=\{(Y,\phi,\psi)|Y\in \mathrm{FM}(X), \phi:S\simeq S_Y, \psi:T\simeq T_Y\}.$$
We have the diagram:
\begin{equation}\label{eDiagramFM}
\begin{tikzcd}
    \widetilde{\mathrm{FM}(X)} \arrow[r, "L_\bullet"] \arrow[d, "\Pi"] & \calM_{S,T}\arrow[ld, "\Pi'"]\\
    \mathrm{FM}(X) &
\end{tikzcd}
\end{equation}
The existence of $\Pi'$ follows from the Torelli theorem for cubic fourfolds \cite{voisintorelli}. Thus the map $L_\bullet$ maps distinct fibres of $\Pi$ to disjoint subsets of $\calM_{S,T}$. 
Note that $|\widetilde{\mathrm{FM}}(X)|=24|\mathrm{FM}(X)|.$

Note that we have the following chain of inclusions:
$$S\oplus T\subset L\subset L^\ast\subset S^\ast\oplus T^\ast.$$ This gives the following: $[L:S\oplus T]=[S^\ast\oplus T^\ast :L^\ast]=12.$ 
The overlattice $S\oplus T\subset L$ is determined by the \textit{glueing subgroup} $H:=L/(S\oplus T)\leq D(S)\oplus D(T).$ The projections of $H$ to $D(S)$ and to $D(T)$ are isomorphisms onto their images: we denote these also by $H$. First we claim that $H\cong D(S).$ To see this we note that since $S\subset L$ is saturated, the map from $L^\ast/(S\oplus T)\rightarrow D(S)$ is a surjection. Similarly, $L^\ast/ S\oplus T\rightarrow D(T)$ is an isomorphism since $T$ is saturated in $L$ and the two groups have the same order. Thus $H$ is an index three subgroup of $L^\ast/(S\oplus T)$, and hence $H\cong D(S).$

\begin{lemma}\label{lLatticeCount}
We have
     $|\calM_{S,T}|=24.$
\end{lemma}
\begin{proof}
In order to count the possible lattices $L\in \calM_{S,T}$ we use \cite[Proposition 1.5.1]{nikulin}. In other words, we count the number of primitive embeddings of $S$ into $L$ that have gluing subgroup $H=D(S).$ Such a primitive embedding is determined by the following data:
\begin{itemize}
    \item A subgroup $H'\leq D(T)$ that is anti-isometric to $H$.
    \item an isometry $\gamma \in O(H).$
\end{itemize}
The isometry group of $H$ is $\GL_2(\bZ/2)\times \bZ/2\bZ,$ thus there are $12$ possibilities for $\gamma.$ Finally, there are two possible subgroups of $D(T)$ that are anti-isometric to $H$ - these correspond to choosing one of the two $\bZ/3\bZ$ factors of $D(T)$. Note that the diagonal copy of $\bZ/3\bZ$ is not anti-isometric to the copy of $\bZ/3\bZ\leq H.$  Thus the number of overlattices is 24.
\end{proof}

Let $G_S = S_3 \times \mathbb{Z}/2\mathbb{Z}$ be the isometry group of $S$, and let $G_T = \mathbb{Z}/2\mathbb{Z}$ acting on $T$ via (Hodge) isometries given by global scalings by a sign. Then $G = G_S \times G_T$ acts on $S \oplus T$ in the natural way, and this induces actions on $S^* \oplus T^*$ and 
\[
(S^*/S )\oplus (T^*/T) = D(S) \times D(T) = \bigl(  (\mathbb{Z}/2\mathbb{Z})^2 \times \mathbb{Z}/3\mathbb{Z} \bigr) \times \bigl(  (\mathbb{Z}/2\mathbb{Z})^2 \times \mathbb{Z}/3\mathbb{Z} \times \mathbb{Z}/3\mathbb{Z} \bigr).  
\]
Corresponding to the factors on the right hand side of the previous displayed formula, we will write elements of $D(S) \times D(T)$ as 
\[
((a, b), (a' , b', b'')), \quad a, a' \in ( \mathbb{Z}/2\mathbb{Z})^2 , \: b, b', b'' \in \mathbb{Z}/3\mathbb{Z}
\]
in the sequel. Note that the factor $S_3$ in $G_S$ acts on the copy of $(\mathbb{Z}/2\mathbb{Z})^2$ in $D(S)$ via permuting the three nontrivial elements and fixing the identity, whereas the factor $\mathbb{Z}/2\mathbb{Z}$ in $G_S$ acts by sign changes in the copy of $\mathbb{Z}/3\mathbb{Z}$ in $D(S)$. The action of $G_T$ is just by sign changes in $\mathbb{Z}/3\mathbb{Z} \times \mathbb{Z}/3\mathbb{Z} \subset D(T)$. 
\begin{lemma}\label{lFibresLBullet}
    $L_{Y,\phi_1,\psi_1}\neq L_{Y, \phi_2, \psi_2}$ for $(\phi_i, \psi_i)\in G_S \times G_T$ unless $(\phi_1, \psi_1) = \pm (\phi_2, \psi_2)$. Hence each nonempty fibre of $L_{\bullet}$ has cardinality $2$. 
\end{lemma}
\begin{proof}
The lattice $L$ is determined by the subgroup $L /(S \oplus T) \subset (S^*/S) \times (T^*/T)$ and we know that the subgroups $L/(S\oplus T)$ arising in this way are of the following form:
\[
H_{\alpha, \beta} = \bigl\{  ((a, b), (\alpha (a)  , \beta (b) , 0)) \, \mid \, a \in ( \mathbb{Z}/2\mathbb{Z})^2, b \in   \mathbb{Z}/3\mathbb{Z} \bigr\}
\]
or 
\[
H'_{\alpha, \beta} = \bigl\{  ((a, b), (\alpha (a)  , 0 , \beta (b))) \, \mid \, a \in ( \mathbb{Z}/2\mathbb{Z})^2, b \in   \mathbb{Z}/3\mathbb{Z} \bigr\}
\]
for some (fixed, but arbitrary) $\alpha \in \mathrm{Aut}((\mathbb{Z}/2\mathbb{Z})^2), \: \beta \in \mathrm{Aut}(\mathbb{Z}/3\mathbb{Z})$. It suffices to prove that if for $(\phi, \psi) \in G_S \times G_T$ we have 
\[
(\phi, \psi) \bigl( H_{\alpha, \beta} \bigr) = H_{\alpha, \beta} \quad \mathrm{or} \quad (\phi, \psi) \bigl( H'_{\alpha, \beta} \bigr) = H'_{\alpha, \beta}  
\]
then $(\phi, \psi) = \pm \mathrm{id}$. First note that under $(\phi, \psi)$ the subgroup $( \mathbb{Z}/2\mathbb{Z})^2 = \{ (a, \alpha (a)) \} \subset ( \mathbb{Z}/2\mathbb{Z})^2 \times ( \mathbb{Z}/2\mathbb{Z})^2 \subset D(S) \times D(T)$ gets mapped to the subgroup 
\[
\{ (\phi (a), \alpha (a)) \} = \{ (a, \alpha ( \phi^{-1} (a))) \}
\]
which is different from the subgroup $\{ (a, \alpha (a)) \}$ unless $\phi \in \mathbb{Z}/2\mathbb{Z} \subset G_S$. Then suppose that $\phi \in \mathbb{Z}/2\mathbb{Z} \subset G_S$ and $\psi \in \mathbb{Z}/2\mathbb{Z} = G_T$. 
Then $(\phi, \psi )$ maps the $\mathbb{Z}/3\mathbb{Z}$-factor of $H_{\alpha, \beta}$ or $H'_{\alpha, \beta}$ into a different subgroup $\mathbb{Z}/3\mathbb{Z}$ of the $3$-torsion subgroup $\mathbb{Z}/3\mathbb{Z} \times \mathbb{Z}/3\mathbb{Z} \times \mathbb{Z}/3\mathbb{Z} $ of $D(S) \times D(T)$ unless $(\phi, \psi) = \pm \mathrm{id}$, which proves the Lemma. 
\end{proof}
 
\begin{proof}[Proof of Theorem \ref{tFMPartners}:]
The diagram \ref{eDiagramFM}, together with Lemma \ref{lFibresLBullet} and Lemma \ref{lLatticeCount}, shows that the number of elements in $\mathrm{FM}(X)$ will be equal to two, provided the map $L_\bullet$ is surjective, and provided that for any $L\in\mathcal{M}_{S,T},$ one can extend the Hodge isometry of $T(Y)\cong T(X)$ to one of the Addington-Thomas lattice $\widetilde{H}(Y)\cong \widetilde{H}(X)$ for any corresponding $Y$. This ensures that the corresponding $Y$ is indeed a Fourier Mukai partner of $X$.

We first prove surjectivity of $L_\bullet$. Let $L\in \calM_{S,T}$, and denote by $I_{21,2}$ the unimodular odd lattice of signature $(21,2)$. By \cite[Cor. 1.13.3]{nikulin}, if $v\in I_{21,2}$ is a class with $v^2=3$, then $v^{\perp} \subset I_{21,2} $ is the unique even lattice of discriminant $3$ and signature $(20,2)$. Thus for any $L\in \mathcal{M}_{S, T}$, $L(-1)$ is isomorphic to this lattice. 
 Using this isomorphism, the lattice $T(-1)$ is a sublattice of $I_{21,2}$, and $T(-1)^\perp$ contains $\langle v \rangle \oplus S$, which has the same discriminant group, hence is equal to $T(-1)^\perp$, so this is the lattice:
$$\left(\begin{matrix}
    3 & 3& 3\\
    3 & 7 & 1\\
    3 & 1 & 7
\end{matrix}\right).$$ In particular, by \cite[Theorem 1.1]{laz09}, there exist a cubic fourfold $Y$ with $\eta:H^4(Y,\bZ)_{prim}\cong L(-1)$ as Hodge structures. 

Next we prove that $Y$ is indeed a Fourier-Mukai partner as claimed. The Addington-Thomas lattice $\widetilde{H}(X)$ is isomorphic to an abstract even unimodular lattice $\widetilde{H}$ of signature $(4,20).$ By \cite[1.14.4]{nikulin}, there is a unique primitive embedding of $T(X)$ into $\widetilde{H}$, up to isometries of $\widetilde{H}$. It thus follows that the Hodge isometry $T(Y)(-1)\cong T(X)(-1)$ extends to an isometry of $\widetilde{H}(Y)\cong \widetilde{H}(X)$, which is a Hodge isometry by construction:

\begin{equation}\label{eDiagramFM}
\begin{tikzcd}
    \widetilde{H}(Y) \arrow[r, "\cong", dashed, "{\exists}"' {below}] & \widetilde{H}(X)\\
    T(Y) \arrow[u, hookrightarrow] \arrow[r, "\cong"]& T(X) \arrow[u, hookrightarrow]
\end{tikzcd}
\end{equation}
 Therefore $Y$ is a Fourier-Mukai partner of $X$, and hence the restrictions of $\eta$ to $S(-1)$ and $T(-1)$ induce a triple $(Y,\phi,\psi ) \in \widetilde{\mathrm{FM}}(X)$ such that $L_\bullet(Y, \phi,\psi)=L$, proving the surjectivity. 
\end{proof}

\section{Gale dual non-syzygetic cubic fourfolds and Lagrangian construction data}\label{sLagrangian}

A general reference for the Gale transform and Gale duality is \cite{EP00}. 

\begin{definition}\label{dGaleDualCubics}
Given a non-syzygetic cubic fourfold $X$ with an equation of the form 
\[ 
\det (M) + L_1L_2L_3 =0,
\] 
we define the \emph{Gale dual cubic} $X'$ by the following procedure. Collecting the coefficients with respect to coordinates $X_0, \dots , X_5$ in $\mathbb{P}^5=\mathbb{P}(V)$ into a matrix, the $9$ linear forms $M_{ij}$ of $M$ and the three $L_i$ can be considered as a linear map
\[
	\C^{12} \xrightarrow{(M_{11}, \dots, M_{33},L_1,L_2,L_3)} \C^6.
\]
The kernel of this linear map is again $6$-dimensional and can be interpreted as $12$ linear forms in dual variables
\[
 	\C^6 \xrightarrow{(M'_{11}, \dots , M_{33}', L_1',L_2',L_3')^t} \C^{12} \xrightarrow{(M_{11}, \dots , M_{33},L_1,L_2,L_3)} \C^6 .
\]
We then define the Gale dual cubic $X'$ by the equation
\[
	\det(M') - L_1'L_2'L_3' = 0.
\]
\end{definition}
We will prove that $X'$ is the nontrivial Fourier-Mukai partner of $X$.

\begin{remark}
    Note that changing the order of the $L_i$ does change the order of the $L_i'$ in the same way, but neither the equation of $X$ or $X'$. 
\end{remark}

We can view $M$ formally as a map
\[
	E \otimes F^* \xrightarrow{M}  V^*
\]
where $E$ and $F^*$ are the three-dimensional row and column spaces of $M$. With this notation, the above diagram becomes 
\[
V^{\perp} \xrightarrow{(M', L_1',L_2',L_3')^t} E \otimes F^* \oplus \mathbb{C}^3  \xrightarrow{(M, L_1,L_2,L_3)} V^* 
\]
where we use the usual notation $V^{\perp}$ for the kernel of a surjective map $W \to V^*$. 

The aim of this section is to provide a linear algebraic dataset that determines a pair of Gale dual cubics. To this end, we start with some linear algebra.
Let $E$ and $F$ be three-dimensional complex vector spaces with orientations $\omega_E \in \Lambda^3 E$ and $\omega_F \in \Lambda^3 F$. We consider $\Lambda^3 (E\oplus F)$ as a symplectic vector space with skew-form given by wedge product. We have $\Lambda^3 (E\oplus F) = U_E \oplus U_F$ with 
\[
	U_E =  (\Lambda^2 E \otimes F) \oplus \Lambda^3 E
	\quadand
	U_F =   (E \otimes \Lambda^2 F)  \oplus  \Lambda^3 F.
\] 
Using the orientations we can identify $U_F = U_E^*$. We make the following definition:

\begin{definition}\label{dRhoLagrangian}
Let $A\subset \Lambda^3 (E\oplus F)$ be a vector subspace, 
$A_E = A \cap U_E$ and 
$A_F = A \cap U_F$. $A$ is called a $\rho$-Lagrangian if 
\begin{enumerate}
\item
$A$ is Lagrangian.
\item
$\dim A_E = \dim A_F =4$. 
\end{enumerate}
We call the data $(A,E,F)$ a $\rho-$Lagrangian data set.
\end{definition}

The aim of this section is to prove the following Proposition:
\begin{proposition}
    A pair $X, X'$ of Gale dual cubic fourfolds determines a $\rho$-Lagrangian data set $(A,E,F)$. Conversely, a $\rho$-Lagrangian data set $(A,E,F)$ determines a pair of Gale dual cubic fourfolds.
\end{proposition}
 
The advantage of this linear algebraic description is that it will establish a link between Gale dual cubics and the geometry of Gushel-Mukai fourfolds that are period partners. In particular, it will allow us to conclude that $X$ and $X'$ are birational and Fourier-Mukai partners using the results from \cite{KuzPerryCatCones}. 

\begin{remark}
    The name \textit{$\rho$-Lagrangians} was chosen because the associated Gushel-Mukai fourfolds, constructed in Section \ref{sGM}, will contain so called $\rho$-planes. The ambient Grassmannian $\Gr (2, V_5)$ contains  two families of planes, and a $\rho$-plane is defined to be a sub-Grassmannian $\Gr(2,V_3)$, where $V_3 \subset V_5$ is a 3-dimensional vector subspace of the five-dimensional vector space $V_5$. The name goes back to Hermann Schubert who developed classical Schubert calculus. The other family of planes are so-called $\sigma$-planes which arise from projective lines in $\P (V_5)$ contained in a chosen hyperplane and passing through a chosen point. Somewhat suggestively, the $\rho$ stands for ``ruled" and the $\sigma$ for ``star".
\end{remark}

Suppose that we are given Gale dual cubics $X, X'$ with equations $\det (M) + L_1L_2L_3 =0$ and $\det (M') - L_1'L_2'L_3' =0$. We will now describe how to construct a $\rho$-Lagrangian from this, given a choice of one of the linear forms $L_1, L_2, L_3$, without loss of generality $L_1$. The corresponding Lagrangian will be $A_1$. 

A pair of Gale dual cubics is the same as a diagram 
\[
	V^\perp \to E\otimes F^* \oplus \C^3 \to V^*.
\]
One can choose splittings such that this becomes 
\[
	\mathbb{C}^4 \oplus \C \oplus \C \xrightarrow{\widehat{P}} (E\otimes F^* \oplus \mathbb{C}) \oplus \mathbb{C} \oplus \C  \xrightarrow{\widehat{Q}^t} \mathbb{C}^4 \oplus \C \oplus \C
\]
with $\widehat{Q}$ and $\widehat{P}$ both $12\times 6$ matrices of the form 
\[
\widehat{Q} = \begin{pmatrix}
	Q_2 & Q_4& Q_3 \\
	0& 1& 0 \\
	0 & 0 & 1
\end{pmatrix}
\quad \text{ and }\quad
\widehat{P} = \begin{pmatrix}
	P_1 & P_3 & P_4 \\
	0& 1& 0 \\
	0 & 0 & 1
\end{pmatrix}.
\]
The indexing in $\widehat{Q}$ and $\widehat{P}$ will become clearer and prove advantageous below. Notice that by assumption
\[ 
\widehat{Q}^t\widehat{P}
=0.
\]
Let $A_1$ be the image of the map $\C^4 \oplus \C^ 4 \oplus \C \oplus \C \to \Lambda^3 (E \oplus F) = U_E \oplus U_F$ given by 
\[
\begin{pmatrix}
Q_1 & Q_2 & Q_3 & Q_4 \\
P_1 & P_2 & P_3 & P_4 
\end{pmatrix} = \begin{pmatrix} Q \\ P \end{pmatrix}
\]
where $Q_1=P_2:=0$. Denote by 
\[
\alpha=\frac{1}{2}(Q^tP -P^tQ), \quad \sigma= \frac{1}{2}(Q^tP +P^tQ)
\]
the natural anti-symmetric and symmetric maps $A_1 \to A^*_1$, identifying $U_F = U_E^*$. Note that $\alpha$ corresponds to the restriction to $A_1$ of the anti-symmetric form given by wedge product on $\Lambda^3 (E \oplus F) $. 

\begin{proposition}\label{pRhoLagrangian}
The subspace $A_1$ is a $\rho$-Lagrangian, and we have 
\[
\sigma = \begin{pmatrix}
0 & 0 & 0 & 0\\
0 & 0 & 0 & 0\\
0 & 0 & 0 & -1\\
0 & 0 & -1 & 0
\end{pmatrix}.
\]
\end{proposition}

\begin{proof}
We have
\[ 
0= \widehat{Q}^t\widehat{P}
=
\begin{pmatrix}
	Q_2^t & 0& 0 \\
	Q_4^t& 1 & 0 \\
	Q_3^t & 0 & 1
\end{pmatrix}
\begin{pmatrix}
	P_1 & P_3 & P_4 \\
	0&  1& 0 \\
	0 & 0 & 1
\end{pmatrix}
= 
\begin{pmatrix}
	Q_2^tP_1 & Q_2^tP_3 & Q_2^tP_4 \\
	Q_4^tP_1 & Q_4^tP_3+1 & Q_4^tP_4 \\
	Q_3^tP_1 & Q_3^tP_3 & Q_3^tP_4+1 \\
\end{pmatrix}
\]
and therefore 
\[
Q^tP = \begin{pmatrix}
 0 & 0 & 0 & 0 \\
 Q^t_2P_1 & 0 & Q_2^t P_3 & Q_2^t P_4 \\
  Q^t_3P_1 & 0 & Q_3^t P_3 & Q_3^t P_4 \\
   Q^t_4P_1 & 0 & Q_4^t P_3 & Q_4^t P_4 
\end{pmatrix} =
\begin{pmatrix}
0 & 0 & 0 & 0\\
0 & 0 & 0 & 0\\
0 & 0 & 0 & -1\\
0 & 0 & -1 & 0
\end{pmatrix}.
\]
Therefore $\alpha=0$ and $\sigma$ as claimed. 
\end{proof}

Conversely, suppose we are given a $\rho$-Lagrangian $A \subset \Lambda^3 (E\oplus F) = U_E \oplus U_F$. We can choose a splitting
\[
A = A_E \oplus A_F \oplus \mathbb{C} \oplus \mathbb{C}
\]
such that the map given by the inclusion
\[
A_E \oplus A_F \oplus \mathbb{C} \oplus \mathbb{C} \to U_E \oplus U_F
\]
is of the form 
\[
\begin{pmatrix}
Q_1 & Q_2 & Q_3 & Q_4 \\
P_1 & P_2 & P_3 & P_4 
\end{pmatrix} = \begin{pmatrix} Q \\ P \end{pmatrix}
\]
with $Q_1=P_2=0$. 
Since $A$ is a Lagrangian, $\alpha =0$, and hence $Q^tP$ is symmetric and therefore 
\[
\sigma= 
Q^tP = \begin{pmatrix}
 0 & 0 & 0 & 0 \\
 Q^t_2P_1 & 0 & Q_2^t P_3 & Q_2^t P_4 \\
  Q^t_3P_1 & 0 & Q_3^t P_3 & Q_3^t P_4 \\
   Q^t_4P_1 & 0 & Q_4^t P_3 & Q_4^t P_4 
\end{pmatrix} =
\begin{pmatrix}
0 & 0 & 0 & 0\\
0 & 0 & 0 & 0\\
0 & 0 &  Q_3^t P_3 & Q_3^t P_4\\
0 & 0 & Q_4^t P_3 & Q_4^t P_4 
\end{pmatrix}
\]
with $Q_4^tP_3 = Q_3^tP_4$. After a change of basis in $\C \oplus \C$ we can assume 
\[
	(Q_3,Q_4)^t (P_3,P_4) = \begin{pmatrix} 0 & -1 \\ -1 & 0 \end{pmatrix}. 
\]
Since $A$ is $\rho$-Lagrangian, $Q_1=P_2=0$. Then reading the proof of Proposition \ref{pRhoLagrangian} backwards, we get $\widehat{Q}^t\widehat{P}=0$. Thus we obtain Gale dual cubics defined by $\widehat{P}, \widehat{Q}$. 

\begin{remark}\label{rSameLagrangians}
Note that interchanging the roles of $X$ and $X'$ in the above discussion interchanges $P$ and $Q$, and $E$ and $F$ at the same time, and therefore leads to the same set of Lagrangians $A_i$. 
\end{remark}

\section{Representation-theoretic intermezzo}\label{sRepTheory}

Let $V_6$ be a six-dimensional complex vector space and suppose that $V_6=E\oplus F$ with $\dim E=\dim F=3.$ The group $G_{3,3}=\SL(E)\times SL(F)$ acts on $V_6$ with this splitting.




In the sequel we fix a basis $e_1, e_2, e_3, f_1, f_2, f_3$ of $V_6^*$ such that $E^*  =\langle e_1,e_2, e_3\rangle$,  $F^* = \langle f_1, f_2, f_3\rangle$. 

We let $R= \mathbb{C}[\Lambda^3 V_6]$. We are interested in the invariant ring $R^G$ for $G:=G_{3,3}$ one of the groups above. We only need the computation up to degree $3$ in the sequel. 

\newcommand{\hate}{\widehat{e}}
\newcommand{\hatf}{\widehat{f}}
\newcommand{\hatu}{\widehat{u}}

\begin{definition}\label{dInvariants33}

Define $\hate_1, \hate_2, \hate_3 \in \Lambda^2 E^*$  and $\hatf_1,\hatf_2,\hatf_3 \in \Lambda^2 F^*$ by 
\[
    \hate_1 = e_2 \wedge e_3, \quad 
    \hate_2 = -e_1 \wedge e_3, \quad
    \hate_3 = e_1 \wedge e_2 
\]
 
and similarly for $\hatf_i$. Set
\[
L_E :=e_1\wedge e_2\wedge e_3, \quad
L_F :=-f_1\wedge f_2\wedge f_3,
\]
and 
\[
M_E := (\hate_i \wedge f_j), \quad 
M_F := (e_i \wedge \hatf_i).  
\]
We interpret $M_E$ and $M_F$ either as a $3 \times 3$ matrix or, according to context, as a $1 \times 9$ vector of linear forms, ordering the matrix entries lexicographically. 
\end{definition}

\begin{theorem}\label{tInvariants}
$R^{G_{3,3}}$ is generated up to degree three by $L_E,L_F,\tr(M_E M_F^t), \det M_E, \det M_F$.
\end{theorem}

\begin{proof}
 See \cite[Sections5-8.m2]{MacaulayFiles}.
\end{proof}

\begin{definition}\label{d10Tuples}
Consider the $10$-tuples of linear forms on $\Lambda^3 V_6$
\[
    (M_E,L_E) 
     =: (u_1,\dots,u_{10}),
     \quad
     (M_F,L_F) 
     =: (\hatu_1,\dots,\hatu_{10}). 
\]
We call
\[
        \sigma = \sum_{i=1}^{10} u_i \cdot \hatu_i \in R
\]
the \emph{$\sigma$-quadric} in what follows. 
\end{definition}

\begin{proposition}\label{pDualBases}
The $10$-tuples $(M_E,L_E)$ and $(M_F,L_F)$ 
are dual bases of
\[
    \Lambda^3 E^* \oplus \bigl(\Lambda^2 E^* \otimes F^*\bigr)  
    \quadand
    \bigl( E^* \otimes \Lambda^2 F^* \bigr) \oplus \Lambda^3 F^* 
\]
with respect to the pairing given by the wedge product and the orientation
$L_E \wedge L_F$. Furthermore, $\sigma$ is $G_{3,3}$-invariant and
\[
    \sigma = \tr(M_E M_F^t) + L_EL_F.
\]
\end{proposition}

\begin{proof}
Note that $\hate_i \wedge e_i = L_E$ and $f_i \wedge \hatf_i = -L_F$ for $i=1,2,3$ and $\hate_i \wedge e_j = 0, f_i \wedge \hatf_j = 0$ for $i \not=j$. With this we have
\[
     (\hate_i\wedge f_j) \wedge (e_i\wedge \hatf_j) 
      = -\hate_i\wedge e_i \wedge f_j \wedge \hatf_j 
      = L_E \wedge L_F
\]
and all other wedge products vanish. Hence $(M_E,L_E)$ and $(M_F,L_F)$ are dual bases as claimed. We compute
\[
    \tr(M_E M_F^t) = \sum_{i=1}^3 \sum_{j=1}^3(\hate_i\wedge f_j) \cdot (e_i\wedge \hatf_j)
\]
and the formula for $\sigma$ follows. 
\end{proof}

\begin{definition}\label{dBigCubics}
We denote by $\widetilde{X}_E,\widetilde{X}_F \subset \P(\Lambda^3 V_6)$ the cubic
hypersurfaces defined by
\[
    2\det M_E - \sigma L_E = 0 
    \quadand
    2\det M_F  + \sigma  L_F = 0.
\]
\end{definition}

\begin{proposition}\label{pProjectionLagrangian}
    Let $A \subset \Lambda^3 V_6$ be a $\rho$-Lagrangian. Then $\widetilde{X}_E \cap A$ is a cone with vertex $\P(A_F)$ and $\widetilde{X}_F\cap A$ is a cone with vertex $\P(A_E)$. Let 
    \[
        X_E \subset \P(A_E\oplus \C \oplus \C) 
        \quadand
    X_F \subset \P(A_F\oplus \C \oplus \C)
    \]
    denote the projections from the respective vertices. Let $\iota \colon \C^2 \to \C^2$
    be the involution given by the matrix 
    $\left(\begin{smallmatrix} 0 & 1 \\ 1 & 0 \end{smallmatrix}\right).$
    Then $\iota(X_E)$ and $X_F$ are Gale dual cubics in the sense of Definition \ref{dGaleDualCubics}. 
\end{proposition}    

\newcommand{\toAdual}{\pi}
    
\begin{proof} 
Let
\[
    \toAdual \colon \C[\Lambda^3 V_6] \to \C[A]
\]
be the ring homomorphism induced by the inclusion $
A \subset \Lambda^3 V_6$.

Consider the basis $(M_E,L_E,M_F,L_F)$ of $\Lambda^3 V_6^*$. Since $A$ is a $\rho-$Lagrangian, there exists a basis $a_0,\dots,a_9$ of $A^*$ such that 
$\toAdual \colon \Lambda^3 V_6^* \to A^*$
is represented by a $10 \times 20$ matrix
\[
    \begin{pmatrix}
        0 & Q_2 & Q_3 & Q_4 \\
        P_1 & 0 & P_3 & P_4 
    \end{pmatrix}^t  = (Q,P)
\]
with $Q^tP - P^tQ = 0$ since $M_E,L_E$ and $M_F,L_F$ are dual bases with respect to the wedge product.
The zeros in this matrix mean that
\[
    \pi(M_E,L_E) \subset \langle a_4,\dots,a_9 \rangle
    \quadand
     \pi(M_F,L_F) \subset \langle a_0,\dots,a_3,a_8,a_9 \rangle.
\]
The coefficient matrix of $\pi(M_E,L_E)$ is $(Q_2,Q_3,Q_4)$ and the coefficient matrix of $\pi(M_F,L_F) = (P_1,P_3,P_4)$.

As computed in the last section, we can also assume
\[
Q^tP = 
\begin{pmatrix}
0 & 0 & 0 & 0\\
0 & 0 & 0 & 0\\
0 & 0 &  0 & -1\\
0 & 0 & -1 & 0 
\end{pmatrix}
\]
which means
\[
    \toAdual(\sigma) = -a_8a_9-a_9a_8 = -2a_8a_9
\]
in our current setting.

The equations of $X_E$ and $X_F$ are  
\[
    \toAdual(\det M_E) + \toAdual(L_E)a_8a_9 = 0
    \quadand
    \toAdual(\det M_F) - \toAdual(L_F)a_8a_9 = 0
\]
and therefore the $12 \times 6$ coefficient matrix of $\iota(X_E)$
is
\[
\widehat{Q} = \begin{pmatrix}
	Q_2 & Q_4& Q_3 \\
	0& 1& 0 \\
	0 & 0 & 1
\end{pmatrix}
\]
and the $12 \times 6$ coefficient matrix of $X_F$ is
\[
\widehat{P} = \begin{pmatrix}
	P_1 & P_3 & P_4 \\
	0& 1& 0 \\
	0 & 0 & 1
\end{pmatrix}.
\]
By the computations of the previous section we have $\widehat{Q}^t\widehat{P} = 0$ and therefore $\iota(X_E)$ and $X_F$ are Gale dual cubics.
\end{proof}

\section{Fano varieties of lines and EPW sextics}\label{sFanoLines}
We've seen that Gale dual cubic fourfolds $X, X'$ determine a $\rho$-Lagrangian data set. One can associate a so called EPW-sextic to such data, in the following way:

\begin{definition}\label{dEPW}
For a Lagrangian $A\subset \Lambda^3 (E\oplus F)$, let
$$Y_{A^\perp}:=\{V_5\in \bP(V_6^*) \mid \dim(A\cap \wedge^3 V_5)\geq 1\}\subset \bP(V_6^*)$$ 
be the associated EPW-sextic. 
\end{definition}
The Gale dual cubic fourfolds determine their Fano varieties of lines $F(X), F(X'),$ which are hyperk\"ahler fourfolds. Similarly, the EPW sextic $Y_{A^\perp}$ determines a hyperk\"ahler fourfold, by taking O'Grady's double cover $\widetilde{Y_{A^\perp}}$. The aim of this section is to prove the following:
\begin{theorem}
    Let $X, X'$ be suitably general Gale dual cubic fourfolds, where $X$ has equation:
    $$\det M+L_1L_2L_3=0.$$ Let $(A_i,E,F)$ be the $\rho$-Lagrangian data set constructed in the previous section.
    Then both $F(X)$ and $F(X')$ are birational to $\widetilde{Y_{A_i^\perp}}$. 
\end{theorem}

In particular, $F(X)$ and $F(X')$ are birational. This recovers several of the birational models found in \cite[Theorem 1.2]{BFMQ24} explicitly.

\begin{definition}\label{dP2}
Let $X \subset \P^5$ be a cubic $4$-fold with equation
\[
	\det M + L_1L_2L_3 =0
\]
and $(e,f) \in E^* \oplus F^*$. If $e \not=0$ 
we denote by  $\Pi_i(e,f) \subset \P^5$ the linear subspace given by $$Mf+e L_i = 0$$
and by  $\Gamma_i (f) \subset \Pi_i(e,f)$ the linear subspace  given by $Mf = 0$ and $L_i=0$. Notice that $\Gamma_i(f) \subset X$.
\end{definition}

Notice that generically $\Pi_i(e,f)$ is a plane and    $\Gamma_i(f)$ is a line. 

\begin{proposition}\label{pAllPlanes}
 Let $\Pi \subset \P^5$ be a plane containing a line $\Gamma_i(f)$. Assume that $L_i$ does not vanish on $\Pi$. Then 
there exists a unique $e$ such that $\Pi = \Pi_i(e,f)$. 
\end{proposition}

\begin{proof}
Let $Mf= (m_1, m_2, m_3)^t$ be a generalized column of $M$. The ideal of the line is $(m_1,m_2, m_3, L_i)$. The ideal of the plane $\Pi$ is a three-dimensional subspace, which we can write as 
\[
(m_1 + e_1L_i, m_2 + e_2L_i, m_3 + e_3L_i)
\]
with $e =(e_1,e_2, e_3) \in E$ since $L_i$ does not vanish on $\Pi$. 
\end{proof}

\begin{definition}\label{dConic}
Given $(e,f) \in E^* \oplus F^*$, $e \not=0$ such that $\Pi_i(e,f)$ is a plane and    $\Gamma_i(f)$ is a line, we denote by $C_i(e,f)$ the conic residual to the line $\Gamma_i (f)$ in $X \cap \Pi_i(e,f)$.
\end{definition}

\begin{proposition}\label{pSingular}
$C_i(e,f)$ is singular if  $(e,f)$ is a point on the EPW-sextic $Y_{A_i^\perp}$. 
\end{proposition}

\begin{proof}
As in the proof of Proposition \ref{pProjectionLagrangian} we consider the
ring homomorphism
\[
    \toAdual \colon \C[\Lambda^3 V_6] \to \C[A_i]
\]
induced by the inclusion $A_i \subset \Lambda^3 V_6$. We define the following varieties in $\P(\Lambda^3 V_6)$:
\begin{enumerate}[label={(\arabic*)}]
    \item $\widetilde{\Pi}_i(e,f) =  \{ M_E f + eL_E = 0 \}$
    \item $\widetilde{\Gamma}_i(f) =  \{ M_E f  = L_E = 0 \}$ 
    \item $\widetilde{C}_i(f)$ the variety residual to $\widetilde{\Gamma}_i(f)$ in $\widetilde{X}_E \cap \widetilde{\Pi}_i(f)$
\end{enumerate}
Note that in all three cases the intersection with $\P(A_i)$ is a cone with vertex containing $\P(A_F)$ and the projections from $\P(A_F)$ to $\P(A_E \oplus \C \oplus \C)$
are $\Pi_i(e,f), \Gamma_i(f)$ and $C_i(e,f)$ as defined above. 

If $e$ and $f$ are nonzero, we can assume $e=e_1$ and $f=f_1$ using the notation from the previous section. Let $V_5 \subset V_6$ be the hyperplane defined by $e+f \in V_6^*$. 

If $(e,f)$ is a point on the EPW-sextic, then $\Lambda^3 V_5 \cap A_i$
is not empty. Let $x \not= 0$ be a point in the intersection. 
One can compute \cite[Sections5-8.m2]{MacaulayFiles} that $\P(\Lambda^3 V_5) \subset \widetilde{C}_i(e,f)$. This means in particular that the projection of $x$ to $\P(A_E \oplus \C \oplus \C)$ lies in $C_i(e,f)$. 

Now let $\widetilde{H}_x= \{ y \in \Lambda^3 V_6 | x \wedge y =0 \}$. One can compute that $\widetilde{H}_x$ is tangent to $\widetilde{C}_i(e,f)$ in $x$. Since $A_i$ is Lagrangian and $x\in A_i$, we have $\P(A_i) \subset H_x$ and therefore $C_i (e,f)=  \widetilde{C}_i(e,f) \cap \P(A_i)$ is singular in $x$.
\end{proof}

\begin{corollary}\label{cBirFanos}
Assume that there exists a point  $(e,f)$ on the EPW-sextic $Y_{A_i^\perp}$ such that $\Pi_i(e,f)$ is a plane and    $\Gamma_i(f)$ is a line.
The construction in Proposition \ref{pSingular} induces a birational map 
\[
\widetilde{Y_{A_i^\perp}} \dashrightarrow F(X)
\]
where $\widetilde{Y_{A_i^\perp}}$ is the O'Grady double cover of $Y_{A_i^\perp}$. Replacing $X$ by $X'$ we get that $F(X)$ and $F(X')$ are birational. 
\end{corollary}

\begin{proof}
By  \cite[Remark 3.3 and comments after Prop. 3.7]{BFMQ24} we know that the birational involution on $F(X)$ given by interchanging the two lines in the singular conic $C_i(e,f)$ coincides with the covering involution of the O'Grady double cover $\widetilde{Y_{A_i^\perp}}$. Thus the above construction gives a well-defined map $\widetilde{Y_{A_i^\perp}} \dashrightarrow F(X)$. To see that this is generically one-to-one onto its image, consider a general line $\ell$ in the image. This intersects $L_i=0$ in a point, and the matrix $M$ drops rank. Let $f$ span the kernel of $M$. Let $\Pi$ be the span of $\Gamma_i(f)$ and $\ell$. By Proposition \ref{pAllPlanes} there exists a unique $e$ such that $\Pi = \Pi_i(e,f)$. Since $F(X)$ and $Y_{A_i^\perp}$ are four-dimensional, $\widetilde{Y_{A_i^\perp}} \dashrightarrow F(X)$ is birational.
\end{proof}

\section{The connection with Gushel-Mukai fourfolds}\label{sGM}
The purpose of this section is to establish a link between the Gale dual cubics of Section \ref{sLagrangian} and Gushel Mukai fourfolds. This will allow us to show that for a very  general cubic fourfold $X\in \calC_{nonsyz},$ the Gale dual cubic $X'$ is the only non-trivial Fourier-Mukai partner of $X$. We will also show that $X, X'$ are birational and non-isomorphic, and their Fano varieties of lines are birational. In particular, $X'$ is the cubic fourfold of \cite[Theorem 3.1]{Brooke:2024aa}.

We start by recalling a few facts about the geometry of Gushel-Mukai fourfolds needed in the sequel, following mainly \cite{KuzPerry} and \cite{DK1}. Let $V_5$ be a vector space of dimension $5$.

\begin{definition}
    An (ordinary) Gushel-Mukai fourfold (GM fourfold) $Z$ is a smooth transverse intersection of the form $$Z:=Gr(2, V_5)\cap \bP(W)\cap Q\subset \bP(\wedge^2V_5),$$ where $W\subset \wedge^2 V_5$ is a codimension 1 subspace, $Q$ is a quadric.
\end{definition}
By \cite[Thm. 3.10]{DK1}, a smooth GM fourfold $Z$ as above is determined uniquely by a \textit{Lagrangian data set} $(V_6,V_5, A)$ where $V_6$ is a 6-dimensional vector space, $V_5\subset V_6$ is a hyperplane, and $A\subset \wedge^3V_6$ is a Lagrangian subspace containing no decomposable vectors with $\dim (A\cap \Lambda^3 V_5)=1$.  Conversely, such a triple $(V_6, V_5, A)$ determines a smooth GM fourfold. For given $Z$, $V_5$ can be identified with the vector space of quadrics $H^0 (\bP(W), \mathcal{I}_M(2))$ that vanish on $M = Gr(2, V_5)\cap \bP(W)$, noting that $M$ is cut out by the traces of the Pl\"ucker quadrics
\[
q(v)(w, w):= v\wedge w\wedge w
\]
for $v\in V_5$, $w\in W$ (after choosing an isomorphism $\Lambda^5 V_5 \simeq \bC$). Then $V_6=V_6(X)$ is the space of all quadrics $H^0 (\bP(W), \mathcal{I}_Z(2))$ vanishing on $Z$. 

For any Lagrangian $A\subset \wedge^3 V_6$, recall that according to Definition \ref{dEPW} we get $Y_{A^\perp}$. By O'Grady's work, when $A$ contains no decomposable vectors, this is an EPW sextic. Let $\calM^{GM}$ be the moduli space of GM fourfolds, and let $\calM^{EPW}:= \mathrm{LGr}(\wedge^3 V_6)//\PGL(V_6)$ be the GIT moduli space for EPW sextics, constructed by O'Grady \cite{OG16}. We have a map
$$\pi:\calM^{GM}\twoheadrightarrow \calM^{EPW}$$ which associates to a GM fourfold the corresponding Lagrangian $[A]$. By \cite{DK3}, the fiber $\pi^{-1}([A])$ of a general Lagrangian with no decomposable vectors is isomorphic to $Y_{A^\perp}$ (modulo the finite group of linear automorphisms of $V_6$ preserving $A^\perp$).

If $V_3 \subset V_5$ is a three-dimensional subspace such that $\Pi:=\mathrm{Gr}(2, V_3) \subset Z$, then we call $\Pi$ a $\rho$-plane of the Gushel-Mukai fourfold $Z$.

\begin{theorem}\label{tConnectionToGMs}
Let $(A, E, F)$ be a $\rho$-Lagrangian data set such that $A$ contains no decomposable vectors.
\begin{enumerate}
\item
The EPW sextic $Y_{A^{\perp}}$ contains two disjoint rational surfaces $\Sigma$, $\Sigma' $ isomorphic to $\bP^2$ whose points correspond to Gushel-Mukai fourfolds containing $\rho-$planes. Conversely, a general EPW-sextic of this type arises from a $\rho$-Lagrangian data set in this way. Every Gushel-Mukai fourfold in $\Sigma$ contains a distinguished $\rho-$plane $\mathrm{Gr}(2, E)$ and every Gushel-Mukai in $\Sigma'$ contains a distinguished $\rho-$plane  $\mathrm{Gr}(2, F)$.

\item 
Given an EPW sextic $Y_{A^{\perp}}$ as in part a), let $Z$ and $Z'$ be any Gushel-Mukai fourfolds corresponding to points in $\Sigma, \Sigma'$. Projecting $Z$ from $\mathrm{Gr}(2, F)$ gives $X$ and projecting $Z'$ from $\mathrm{Gr}(2, E)$ gives $X'$, where $X$ and $X'$ are the Gale dual cubic fourfolds constructed in Section \ref{sLagrangian}. 
\item 
$X$ and $X'$ are Fourier-Mukai partners and birational. For general choices, the Fano varieties of lines of $X$ and $X'$ are birational, but $X$ and $X'$ are not isomorphic. 
\end{enumerate}
\end{theorem}

\begin{proof}
Let $Z$ be a Gushel-Mukai fourfold, and $(V_6, V_5,A)$ its corresponding Lagrangian data set. For part a) note that by \cite[Remark 5.1]{KuzPerry} the condition for $Z$ to contain a $\rho$-plane (i.e., $\mathrm{Gr}(2, V_3) \subset Z$) is equivalent to the conditions
\begin{equation}\label{eKuzPerry}
\dim \Bigl( A \cap \bigl( (\Lambda^2 V_3)\wedge V_6 \bigr) \Bigr) \ge 4, \quad V_3 \subset V_5 .
\end{equation}
If $V_6 =E \oplus F$ with $\dim E =\dim F=3$ and $A \subset \Lambda^3 V_6$ is a $\rho-$Lagrangian, then 
\[
\dim \Bigl( A \cap \bigl( (\Lambda^2 E)\wedge V_6 \bigr) \Bigr) = 4, \: \dim \Bigl( A \cap \bigl( (\Lambda^2 F)\wedge V_6 \bigr) \Bigr) = 4
\]
by Definition \ref{dRhoLagrangian}. The subspaces $V_5 \supset E$ define a $\bP^2$ inside the EPW sextic $Y_{A^{\perp}}$, which is $\Sigma:=\bP(E^{\perp}) \subset Y_{A^{\perp} }\subset \bP(V_6^*)$. Similarly, interchanging $E$ and $F$, we get $\Sigma':= \bP(F^{\perp}) \subset Y_{A^{\perp} }\subset \bP(V_6^*)$. Clearly, $\Sigma$ and $\Sigma'$ are disjoint since there is no $V_5$ containing both $E$ and $F$. 

Again by \cite[Remark 5.1]{KuzPerry}, having fixed $A$, the $V_3$'s satisfying the first part of equation \ref{eKuzPerry} form a finite set. Therefore, the points $[V_5]$ of $Y_{A^{\perp}}$ corresponding to Gushel-Mukai fourfolds containing $\rho$-planes form a dense open subset of finitely many rational surfaces in $Y_{A^{\perp}}$, each of which is isomorphic to $\bP(E_i^{\perp})$ for some uniquely determined three-dimensional subspace $E_i \subset V_6$. If two such surfaces are disjoint, the corresponding linear subspaces, $E$ and $F$ say, must span $V_6$: $V_6 =E\oplus F$ (since otherwise there is a $V_5\subset V_6$ containing both). If $Y_{A^{\perp}}$ is general, the first condition in equation \ref{eKuzPerry} will be an equality, and applying it to $E$ and $F$ we retrieve the definition of a $\rho-$Lagrangian $A$. 

\medskip

Next we prove b). Let $Y_{A^\perp}$ be as in part a), and $Z,Z'$ Gushel-Mukai fourfolds corresponding to points in $\Sigma, \Sigma'$ respectively. We write $V_5:=V_5(Z)$ and $V_5':=V_5(Z')$. 

Following \cite[Section 7.2]{DIM15} and \cite[Section 5]{KuzPerry}, the linear projection of $Z$ from the plane $\Gr(2,F)$ (resp $Z'$ from $\Gr(2,E)$) gives birational maps to cubic fourfolds, which we want to show are nothing but the Gale dual fourfolds $X$ and $X'$ introduced above. We are now given Lagrangian construction data $(V_6, V_5, A)$ for $Z$ (and similarly for $Z'$). Recall that this data determines a map
\[
q=q_{\mathrm{GM}}\colon V_6 \to S^2 A^*
\]
which, putting $W = A /(A \cap \Lambda^3 V_5)$, is such that it factors over $S^2 W^*$. If $\lambda$ is a linear form in $V_6^*$ with kernel $V_5$, then, by the Leibniz rule, contraction with $\lambda$ induces maps
\[
\lambda_p \colon \Lambda^p V_6 \to \Lambda^{p-1} V_5
\]
with kernel $\Lambda^p V_5$. In particular, the composition
\[
\xymatrix{
A \ar@{^{(}->}[r] & \Lambda^3 V_6 \ar[r]^{\lambda_3} & \Lambda^2 V_5
}
\]
induces an injection
\[
\xymatrix{
W = A/(A \cap \Lambda^3 V_5) \ar@{^{(}->}[r] & \Lambda^2 V_5 .
}
\]
Moreover, if $v\in V_5$ and $w\in W$ then 
\[
q(v)(w,w)= v\wedge w \wedge w . 
\]
One then recovers $Z$ as 
\[
Z = \bigcap_{v \in V_6} Q(v)
\]
where $Q(v) \subset \bP(W)$ is the quadric defined by $q(v)$. Note that the preimage of $\Lambda^2 F$ under the composition
\[
\xymatrix{
A \ar@{->>}[r] & W = A/(A \cap \Lambda^3 V_5) \ar@{^{(}->}[r] & \Lambda^2 V_5
}
\]
is just  $A \cap \bigl( (\Lambda^2 F)\wedge V_6 \bigr) = A \cap \bigl( (\Lambda^2 F)\wedge (E \oplus F) \bigr) = A_F$. We denote by $\widetilde{Z}\subset \bP (A)$ the preimage of $Z$ in $\bP(A)$, in other words
\[
\widetilde{Z} = \bigcap_{v \in V_6} \widetilde{Q}(v)
\]
where $\widetilde{Q}(v) \subset \bP(A)$ is the quadric defined by $q(v)$ (now in $\bP(A)$). 
We obtain $Z$ by projecting $\widetilde{Z}$ from the point $[(A \cap \Lambda^3 V_5)]$. 
Thus projecting $Z$ from the $\rho-$plane $\mathrm{Gr}(2, F)$ is the same as projecting $\widetilde{Z}$ from the subspace corresponding to $A_F$, landing in $\bP(A/A_F) =\bP (A_E \oplus A_2)$. 
We need to check that as the image of the projection we obtain indeed the cubic $X_E$ constructed from the same $\rho-$Lagrangian data in Section \ref{sRepTheory}. We consider $V_6^*=V_1^*\oplus V_5^*$ and $G_{1,5}:=\SL(V_1)\times \SL(V_5)$. Let $(v_1, v_2, v_3, v_4, v_5):=(e_2, e_3, f_1, f_2, f_3)$ be a basis of $V_5^*$. Then
\[
N_{1,5} = (v_i \wedge v_j \wedge e_1 )
\]
is a skew-symmetric $5\times 5$ matrix with entries in $\Lambda^3 V_6^*$. Let 
\[
U_{1,5} = (u_1, \dots , u_{10}) = ((-1)^{i-j+1}e_1\wedge v_i \wedge v_j )_{i<j}
\]
and 
\[
\widehat{U}_{1,5} = (\hat{u}_1, \dots , \hat{u}_{10}) = (v_k\wedge v_l \wedge v_m)_{k<l<m}
\]
sorted such that $u_i \wedge \hat{u}_i = e_1\wedge v_1\wedge v_2\wedge \dots \wedge v_5$. Then
\[
\sigma_{1,5}= \sum_{i=1}^{10} u_i \cdot \hat{u}_i
\]
is a $G_{1,5}$-invariant quadric in $\mathbb{C}[\Lambda^3 V_6]$. Let $\widetilde{Z}_{1,5}$ be the variety defined by the $4\times 4$ Pfaffians of $N_{1,5}$ and $\sigma_{1,5}$. Using Macaulay2 \cite[Sections5-8.m2]{MacaulayFiles}, one can compute that the equation of $\widetilde{X}_E$ from Definition \ref{dBigCubics} is contained in the ideal of $\widetilde{Z}_{1,5}$. Since $\widetilde{Z} = \widetilde{Z}_{1,5} \cap \bP (A)$, we get that $X_E$ contains the image of the projection of $Z$ from the $\rho$-plane $\mathrm{Gr}(2, F)$. Since this image is a cubic fourfold by \cite[Section 5]{KuzPerry}, we get that it equals $X_E$. An analogous argument shows that projecting $Z'$ from $\mathrm{Gr}(2, E)$ gives $X_F$. Note that by Proposition \ref{pProjectionLagrangian}, $X=\iota(X_E)$ and $X'=X_F$. 
 
\medskip

We now prove c). The fact that $X$ and $X'$ are Fourier-Mukai partners follows combining two facts. Firstly, we use \cite[Thm. 1.6]{KuzPerryCatCones}, which shows that the GM categories (or Kuznetsov categories) $\mathcal{K}(Z)$ and $\mathcal{K}(Z')$ (defined in \cite[Prop. 2.3]{KuzPerry}) are equivalent if the Gushel-Mukai fourfolds are period partners. This applies to $Z$ and $Z'$ since the associated Lagrangian subspace $A$ is the same in both cases. Secondly, the Kuznetsov components $\mathcal{A}_X$ and $\mathcal{A}_{X'}$ are equivalent to $\mathcal{K}(Z)$ and $\mathcal{K}(Z')$, respectively, by \cite[Thm. 5.8]{KuzPerry}. 

Moreover, $Z$ and $Z'$ are birational because they are period partners, using \cite[Thm. 4.15]{DK1}, and hence $X$ and $X'$ are birational. Corollary \ref{cBirFanos} shows that the Fano varieties of lines $F(X)$ and $F(X')$ are birational if $X$ and $X'$ are general. 

It remains to show that for general choices, $X$ and $X'$ are not isomorphic. If $X$ and $X'$ are isomorphic, they are projectively equivalent. A \emph{very general} nonsyzygetic cubic $X$ has precisely three hyperplane sections that are cubic symmetroids, cubic threefolds with six nodes in linear general position and if 
\[
	\det M + L_1L_2L_3 =0 
\]
is an equation for $X$, then these are given by $\{L_i=0\} \cap X$, $i=1,2,3$. Hence, if $X'$ has equation
\[
	\det M' - L_1'L_2'L_3' =0 
\]
any projectivity mapping $X$ onto $X'$ will carry the hyperplane sections $L_i=0$ onto the hyperplane sections $L_j'=0$ (in some order), assuming that $X$, and with it $X'$, is very general.
Hence such a projectivity will carry the three cubic surfaces $\Sigma_1, \Sigma_2, \Sigma_3$ given as $X \cap \{L_i =L_j=0\}$, $\{i, j\} \subset \{ 1,2,3\}$ a two element subset, onto the corresponding $\Sigma_1', \Sigma_2', \Sigma_3'$ obtained from $X'$ in the same way. 

One can check by a computer algebra calculation that for a \emph{special}  choice of partners $X$ and $X'$, hence in general, none of the $\Sigma_i$ are isomorphic to any $\Sigma_j'$, using the Clebsch-Salmon invariants of cubic surfaces as described in \cite{ElsenhansJahnel} to distinguish non-isomorphic cubic surfaces. Specifically, the code checks \cite[Sections5-8.m2]{MacaulayFiles} that for the cubics $X$ and $X'$ of the family described in Example \ref{example:A4} with parameters $\alpha=\gamma=\delta =\lambda =1$, $\beta =2$, the (absolute) invariant  $I_8^2/I_{16}$ from \cite[Prop. 6.2]{ElsenhansJahnel} is different for the $\Sigma_i$ from the corresponding invariant for any of the $\Sigma_j'$. Since these are polynomial invariants depending continuously on the initially given cubics $X, X'$, we obtain that a very general pair $X, X'$ is not isomorphic since these invariants will still be distinct. From this we can deduce that $X$ and $X'$ will also be nonisomorphic if we only assume $X$ to be \emph{general}, not just \emph{very general}: indeed, the locus of  nonsyzygetic cubic fourfolds inside the moduli space $\mathfrak{M}$ of smooth cubic fourfolds admits a unirational parametrisation by (a dense open subset of) a big affine space $\mathfrak{V}$ (with coordinates all the coefficients of linear forms occurring in the matrix $M$ and $L_1, L_2, L_3$). We obtain two morphisms $f, f'$ from (a dense open subset of) $\mathfrak{V}$ to $\mathfrak{M}$ in this way, where $f$ sends the point in $\mathfrak{V}$ to the isomorphism class of $X$, and $f'$ to that of $X'$. Then $X$ and $X'$ will be nonisomorphic away from the preimage of the diagonal under $f \times f' \colon \mathfrak{V} \to \mathfrak{M}\times \mathfrak{M}$. Since there is at least one nonisomorphic pair $(X, X')$ by the above argument, this preimage is not everything. 
\end{proof}

\section{Non-syzygetic $G$-cubic fourfolds}\label{sGActions}
One of our motivations for the linear algebraic construction of non-syzygetic cubics is to study $G$-cubic fourfolds. Indeed, it is easy to endow both $X$ and the Gale dual cubic $X'$ with a group action using the $\rho$-Lagrangian construction.

\begin{definition}\label{dGLagrangianData}
Let $G$ be a finite group and let $E$ and $F$ be faithful three-dimensional $G$-representations. Suppose $A \subset \Lambda^3 (E\oplus F)$ is a $G$-subrepresentation and at the same time a $\rho-$Lagrangian. We then call $(G, A, E, F)$ a $(G, \rho)-$Lagrangian data set and $A$ a $(G, \rho)-$Lagrangian.
\end{definition}

\begin{proposition}\label{pGCubics}
Let $(G, A, E, F)$ be a $(G, \rho)-$Lagrangian data set. Then $G$ acts naturally on the cubic hypersurfaces $\widetilde{X}_E,\widetilde{X}_F \subset \P(\Lambda^3 V_6)$ from Definition \ref{dBigCubics} given by
\[
    2\det M_E - \sigma L_E = 0 
    \quadand
    2\det M_F  + \sigma  L_F = 0,
\]
and also on their projections \[
        X_E \subset \P(A_E\oplus \C \oplus \C) 
        \quadand
    X_F \subset \P(A_F\oplus \C \oplus \C) .
    \]
    \end{proposition}
\begin{proof}
This is obvious from the construction in Section \ref{sRepTheory} if $G$ is contained in $G_{3,3}=\mathrm{SL}(E) \times \mathrm{SL}(F)$. If more generally, $G\subset \mathrm{GL}(E) \times \mathrm{GL}(F)$, note that the quantities $L_E$, $M_E$ and $\sigma$ defined in Section \ref{sRepTheory} have natural bidegrees given by the degree of the wedge monomials in the $e_i$ and $f_j$ occurring in them. Then $\det M_E$ has bidegree $(6,3)$, $L_E$ bidegree $(3,0)$, $\sigma$ bidegree $(3,3)$. So every term in the equation for $\widetilde{X}_E$ in Definition \ref{dBigCubics} has bidegree $(6,3)$, so the whole equation gets multiplied by $\det(g)^{-2} \det(h)^{-1}$ when we apply an element $(g, h)\in \mathrm{GL}(E) \times \mathrm{GL}(F)$. So $\widetilde{X}_E$ and similarly $\widetilde{X}_F$ are invariant under $G$ and so are their projections $X_E$ and $X_F$ since $A\cap F=A_F$ and $A\cap E =A_E$ are $G$-invariant as well. 
\end{proof}

We will usually restrict attention to the case when the $G-$action on the $G-$cubics $X_E, X_F$ constructed from a $(G, \rho)-$Lagrangian data set is faithful. 

Note that this construction is interesting for testing Conjectures \ref{conj: Huy}, \ref{conj: BFM} from the Introduction in an equivariant set-up: indeed, the proof of the birationality of $X_E$ and $X_F$ in Theorem \ref{tConnectionToGMs} does not necessarily hold equivariantly.  For example, $E$ or $F$ could be irreducible $G$-representations - in this case, the birational map exhibited between $X_E, X_F$ is no longer a $G$-birational map. 

\begin{proposition}\label{pBirationalityConstruction}
    Let $X_E, X_F$ be constructed from a $(G,\rho)$-Lagrangian data set. Assume $E, F\subset V_6$ are reducible $G$-representations with $V_6=E\oplus F$. Suppose that there exist $E\subset V_5, F\subset V_5'$ such that:
\begin{equation}\label{G-reps}
    E=V_1^E \oplus V_2^E, \quad F=V_1^F\oplus V_2^F,
\end{equation} where $V_2^E = V_5'\cap E$ and $V_2^F= F \cap V_5$ and $V_1^E$ and $V_1^F$ are complementary subrepresentations, and $V_1 \subset V_2^E \oplus V_2^F$.
    Suppose that $A$ contains no decomposable vectors. 
    
    Then $X_E$ and $X_F$ are $G$-birational provided that $V_1$ is contained in a specified locally closed non-empty algebraic subset $Y$ of $\bP(V_6)$. 
\end{proposition}
The subset $Y$ above is defined in \cite[Proof of Thm. 4.15 and Appendix (B.2)]{DK1}.
\begin{proof}
We adapt the proof of Theorem \ref{tConnectionToGMs} in order to extend to the $G$-equivariant setting. We will show that, if we choose $E,F$ and $V_1$ as in the statement, then $X_E$ and $X_F$ are $G$-birational.

The first step is ensuring that the Gushel-Mukai fourfolds $Z$ and $Z'$ constructed from the $G$-Lagrangian data set are $G$-Gushel Mukai fourfolds. 
This is guaranteed if one selects the corresponding $V_5$ (and $V_5')$ to be a $G$-representation containing $E$ (respectively $F$).
Next, one needs to construct a $G$-birational map between $Z$ and $Z'$. This can be achieved by adjusting the birationality construction of \cite[Section 4.5-4.6]{DK1}, which we now review and modify to the equivariant context.

In \cite{DK1}, the authors prove that two Gushel Mukai $Z$ and $Z'$ corresponding to the same Lagrangian $A$ are both birational to another smooth Gushel Mukai $\widetilde{Z}$ that is dual to both $Z$ and $Z'$ simultaneously (where by dual we mean as in \cite[Definition 3.26]{DK1}).
The authors construct a birational map between $Z$ and the dual GM fourfold $\widetilde{Z}$ (the same construction holds for $Z')$. This mutual birationality construction depends on a choice of:
$$V_1\subset V_5\cap V_5'$$ 
such that $[V_1]$ is chosen inside a particular locally closed algebraic subset $Y$ of $\bP(V_6)$ whose closure is either  a sextic threefold if $\dim (V_5\cap V_5')=5$ or a sextic surface if $\dim (V_5\cap V_5')=4$.
By the proof of  \cite[Theorem 4.15]{DK1}, this locally closed algebraic subset is non-empty.

In order to make the construction $G$-equivariant, it suffices to pick a $G$-subrepresentation $V_1\subset V_5\cap V_5'$. Thus it follows (without loss of generality) that we must have reducible $G$-representations $E$ and $F$ such that:
\begin{equation}\label{G-reps}
    E=V_1^E \oplus V_2^E, \quad F=V_1^F\oplus V_2^F,
\end{equation} where $V_2^E = V_5'\cap E$ and $V_2^F= F \cap V_5$ and $V_1^E$ and $V_1^F$ are complementary subrepresentations, and $V_1 \subset V_2^E \oplus V_2^F$. In particular, $V_2^E$ or $V_2^F$ or both can be decomposed into a one-dimensional subrepresentation isomorphic to $V_1$ and a one-dimensional complement. 
\end{proof}

    In particular, if $G$ is abelian, its irreducible representations are one-dimensional character representations, and $G$-splittings as in equation \ref{G-reps} always exist. In particular, if $V_2^E \oplus V_2^F = V_1^{\oplus 4}$ for some one-dimensional representation $V_1$, an admissible choice of $V_1$ can always be made because $Y$ is non-empty. 

We also have the following result for general $G$. 

\begin{proposition}\label{prop:FanoLinesBirational}
Let $G$ be a finite group and let $X_E, X_F$ be cubic fourfolds constructed from a $(G, \rho)-$Lagrangian data set. Assume that there exists a point  $(e,f)$ on the EPW-sextic $Y_{A_i^\perp}$ such that $\Pi_i(e,f)$ is a plane and    $\Gamma_i(f)$ is a line. Then the Fano varieties of lines of the two cubics are $G$-birational. 
\end{proposition}

\begin{proof}
The proofs in Section \ref{sFanoLines}, in particular the proof of Corollary \ref{cBirFanos} go through equivariantly, together with the $G$-action. 
\end{proof}

The following is the simplest example when smooth cubics $X_E$ and $X_F$ exist for a nonabelian $G$ where it is not clear to us if $X_E$ and $X_F$ are $G$-birational. Indeed, in this case the birational map of Theorem \ref{tConnectionToGMs} is not $G$-birational.

\begin{example}\label{example:A4}
Let $G=A_4$. Notice that $A_4$ is a semi-direct product $K_4 \rtimes \mathbb{Z}/3$ where $K_4 \simeq \mathbb{Z}/2 \times \mathbb{Z}/2$ is the Klein four subgroup. 
Let $V$ be the three-dimensional irreducible representation of $A_4$ generated by the elements
\[
	\begin{pmatrix}
	1 & 0 & 0 \\
	0 & -1 & 0 \\
	0 & 0  & -1
	\end{pmatrix},
	\begin{pmatrix}
	-1 & 0 & 0 \\
	0 & 1 & 0 \\
	0 & 0  & -1 \\
	\end{pmatrix},
	\begin{pmatrix}
	0 & 1 & 0 \\
	0 & 0 & 1 \\
	1 & 0  & 0
	\end{pmatrix}.
\]	
The first two matrices generate the Klein $4$-group $K_4=K \subset A_4$ whereas the last generates a copy of $C_3=\mathbb{Z}/3$. Let $\xi$ be a primitive cube root of unity and write $\chi_0,\chi_1,\chi_2$ for the three one-dimensional character representations of $A_4$ where the third matrix above acts via $\xi^i$ in $\chi_i$.

\medskip

Choosing $E=F=V$, we get 
\begin{align*}
\Lambda^3 V_6 &= \Lambda^3 E 
\oplus \bigl(\Lambda^2 E \otimes F\bigr) 
\oplus \bigl( E \otimes \Lambda^2 F \bigr) \oplus \Lambda^3 F \\
 &= \chi_0 
 \oplus \bigl( V\oplus V 
 \oplus \chi_0 \oplus \chi_1 \oplus \chi_2 \bigr) 
 \oplus \bigl( V\oplus V 
 \oplus \chi_0 \oplus \chi_1 \oplus \chi_2 \bigr) 
 \oplus \chi_0 .
\end{align*}
Now we pick a four-dimensional subrepresentation  \[ A_E \simeq \chi_0 \oplus V \subset \Lambda^3 E \oplus \bigl(\Lambda^2 E \otimes F\bigr). \]
Notice that we have a 
by choosing $\chi_0$ and a copy of $V$ in the pencil of such $A_4$-subrepresentations that $\Lambda^3 E \oplus \bigl(\Lambda^2 E \otimes F\bigr)$ contains. Then choose 
\[ A_F\simeq \chi_0 \oplus V \subset  \bigl( E \otimes \Lambda^2 F \bigr) \oplus \Lambda^3 F\] such that $A_E \wedge A_F =0.$ Finally, pick $\chi_1$ in the pencil of such representations in $\Lambda^3 V_6$ and $\chi_2$ in $\Lambda^3 V_6$ with bases such that $\sigma$ has the form in Proposition \ref{pRhoLagrangian}. With these choices we obtain the following equations defining  cubics $X_E$ and $X_F$ \cite[Sections5-8.m2]{MacaulayFiles}:
\[
\left(\!\begin{array}{ccc}
       \delta\,X_{7}+\lambda\,(X_{8}+X_{9})
       &\beta\,X_{4}
       &\alpha\,X_{6}
       \\
       \alpha\,X_{4}
       &\delta\,X_{7}+\lambda\,(\xi\, X_{8}+\xi^2\,X_{9})
       &\beta\,X_{5}
       \\
       \beta\,X_{6}
       &\alpha\,X_{5}
       &\delta\,X_{7}+\lambda\,(\xi^2\,X_{8}+\xi\,X_{9})
       \end{array}
\!\right)+\gamma\,X_{7}X_{8}X_{9} =0 
\]
and 
\[
       \left(\!\begin{array}{ccc}
       -\gamma\,X_{3}-\frac{1}{3\,\lambda}(X_{8}+X_{9})
       &-\alpha\,X_{0}
       &\beta\,X_{2}
       \\
       \beta\,X_{0}
       &-\gamma\,X_{3}-\frac{1}{3\,\lambda}(\xi\,X_{8}+\xi^2\,X_{9})
       &-\alpha\,X_{1}
       \\
       -\alpha\,X_{2}
       &\beta\,X_{1}
       &-\gamma\,X_{3}-\frac{1}{3\,\lambda}(\xi^2\,X_{8}+\xi\,X_{9})
       \end{array}\!\right)
       -3\,\delta\,X_{3}X_{8}X_{9} =0 . 
\]
where $(\alpha : \beta)$ represents the choice of $V$ in $A_E$, and $(\gamma : \delta )$ the choice of $\chi_0$ in $A_E$, and $\lambda$ encodes the freedom of choice in $\chi_1$. 

$X_E$ and $X_F$ are invariant under the action of $A_4$ given as follows: the generators above act on $\C[X_0,\dots,X_9]$ with the following matrices:
\begin{align*}
	&\diag\{-1,1,-1,1,-1,1,-1,1,1,1\} \\
	&\diag\{-1,-1,1,1,-1,-1,1,1,1,1\} \\
       	&\begin{pmatrix}
	   0&1&0&\\
       0&0&1\\
       1&0&0\\
       &&&1\\
       &&&&0&1&0\\
       &&&&0&0&1\\
       &&&&1&0&0\\
       &&&&&&&1&&\\
       &&&&&&&&\xi^2&\\
       &&&&&&&&&\xi
       \end{pmatrix}.
\end{align*}

 One can check that $X_E$ and $X_F$ are not isomorphic for specific choices of the parameters over $\mathbb{F}_{97}$ \cite[Example8.5.m2]{MacaulayFiles}.
\end{example}

\begin{question}
Are $X$ and $X'$ always $A_4$-birational? If not, they would provide a counterexample to the equivariant version of Conjecture \ref{conj: BFM} of the Introduction. Are $X$ and $X'$ $A_4$-Fourier-Mukai partners? In other words, are their Kuznetsov components $A_4$-equivalent? If yes and if $X$ and $X'$ are not $A_4$-birational, then they would provide a counterexample to the equivariant versions of Conjectures \ref{conj: Huy} of the Introduction. 
\end{question}

\bibliographystyle{alpha}
\bibliography{bibliography}

\end{document}